\documentclass[reqno,12pt,letterpaper]{amsart}
\usepackage{amsmath,amssymb,mathrsfs,graphicx,amsthm, mathtools}
\usepackage[dvipsnames]{xcolor}
\usepackage[colorlinks=true,linkcolor=Red,citecolor=Green]{hyperref}
\usepackage[T1]{fontenc}
\usepackage[utf8]{inputenc}
\usepackage{tikz-cd}

\usepackage{comment}
\usepackage{subcaption}

\newtheorem{proposition}{Proposition}[section]
\newtheorem{theorem}[proposition]{Theorem}
\newtheorem{lemma}[proposition]{Lemma}
\newtheorem{corollary}[proposition]{Corollary}
\theoremstyle{definition}
\newtheorem{definition}[proposition]{Definition}

\theoremstyle{remark}
\newtheorem{remark}[proposition]{Remark}

\numberwithin{equation}{section}

\newcommand{\N}{\mathbb{N}}
\newcommand{\R}{\mathbb{R}}

\renewcommand{\div}{\mathrm{div}}

\title[On the creation of conjugate points for thermostats]{On the creation of conjugate points for thermostats}
\author[J.~Echevarría Cuesta]{Javier Echevarría Cuesta}
\address{Department of Pure Mathematics and Mathematical Statistics, University of Cambridge, Cambridge CB3 0WB, UK}
\email{je396@cam.ac.uk}

\author[J.~Marshall Reber]{James Marshall Reber\textsuperscript{*}}
\thanks{\textsuperscript{*}Corresponding author}
\address{Department of Mathematics,  University of Chicago, 5734 S. University Avenue, Chicago, IL 60637, USA }
\email{jmarshallreber@uchicago.edu}

\begin{document}

\begin{abstract}

Let $(M, g)$ be a closed oriented Riemannian surface and let $SM$ be its unit tangent bundle. We show that the interior in the $\mathcal{C}^2$-topology of the set of smooth functions $\lambda:  SM\to \R$ for which the thermostat $(M, g, \lambda)$ has no conjugate points is a subset of those functions for which the thermostat is projectively Anosov. Moreover, we prove that if a reversible thermostat is projectively Anosov, then its non-wandering set contains no conjugate points. 
\end{abstract}

\maketitle

\section{Introduction}



Let $(M, g)$ be a closed oriented Riemannian surface and let $\lambda\in \mathcal{C}^\infty(SM,\R)$ be a smooth function on the unit tangent bundle $\pi : SM\to M$. A smooth curve $\gamma : \R\to M$ is a \textit{thermostat geodesic} if it satisfies the second-order differential equation
$$\nabla_{\dot{\gamma}}\dot{\gamma}= \lambda(\gamma, \dot{\gamma})J\dot{\gamma},$$
where $J: TM\to TM$ is the complex structure on $M$ induced by the orientation, that is, rotation by $\pi/2$ according to the orientation of the surface. This equation describes the motion of a particle confined to $M$ and subjected to a force of magnitude $\lambda$ that is always orthogonal to its velocity. Since $\gamma$ has constant speed, this defines a flow on $SM$ given by $\varphi_t(\gamma(0), \dot{\gamma}(0))\coloneq (\gamma(t),\dot{\gamma}(t))$. The infinitesimal generator is $F \coloneq X + \lambda V$, where $X$ is the geodesic vector field and $V$ is the vertical vector field (see \cite[Lemma 7.4]{merry11}). The triple $(M,g, \lambda)$ is called a \textit{thermostat}, reflecting one of its physical interpretations: the surface is subjected to an external force field, which contributes to the kinetic energy, but it is simultaneously kept in contact with a heat reservoir, ensuring that particles on $M$ maintain a constant kinetic energy. 

Thermostats encompass a variety of dynamical systems. For example, we recover geodesic flows when $\lambda=0$, and if $\lambda$ depends only on the position of the particle (that is, $V\lambda=0$), we obtain magnetic flows. When $\lambda$ depends linearly on the velocity, we instead obtain \textit{Gaussian thermostats}. These correspond to the geodesic flows of metric connections on $M$, including those with non-zero torsion (see \cite{przytycki08}). 

A fundamental property of thermostats is that they can be \textit{dissipative}, meaning that they do not preserve any smooth volume form on $SM$. Thus, they provide a unified point of view to study many non-conservative dynamical systems, where the classical Hamiltonian structure is absent.



As in the geodesic case, the exponential map $\exp^\lambda:TM\to M$ is defined by
\begin{equation}\label{eq:exponential-map}
    \exp_x^\lambda(t v)\coloneqq\pi(\varphi_t(x,v)), \qquad x\in M, \, t\geq 0,\, v\in S_xM.
\end{equation}
For every $x\in M$, the map $\exp_x^\lambda$ is $\mathcal{C}^\infty$ on $T_xM\setminus \{0\}$ but, in general, only $\mathcal{C}^1$ at $0$; see, for instance, the proof of \cite[Lemma A.7]{dairbekov07b}. The lack of smoothness at the origin reflects the potential non-reversibility of thermostat flows. We say that the thermostat has \textit{no conjugate points} if $\exp_x^\lambda$ is a local diffeomorphism for all $x\in M$. Given a thermostat geodesic segment $\gamma: [0,T]\to M$ with distinct endpoints $x_0 \coloneq \gamma(0)$ and $x_1\coloneq\gamma(T)$, we say that $x_0$ and $x_1$ are \textit{conjugate along} $\gamma$ if the map $\exp_{x_0}^\lambda$ is singular at $T\dot{\gamma}(0)$, that is, if the differential $d_{T\dot{\gamma}(0)}\exp_{x_0}^\lambda$ has non-trivial kernel.  The no-conjugate-points property means that, on each connected component of $M$, any two points in its universal cover are joined by the lift of a unique thermostat geodesic. We are particularly interested in this dynamical feature because, as shown by Klingenberg \cite{klingenberg74}, having no conjugate points is a necessary condition for a geodesic flow to be Anosov. This remains true for thermostats thanks to \cite[Lemma 4.1]{dairbekov07}.


Recall that the flow $\{\varphi_t\}_{t \in \R}$ is said to be \textit{Anosov} if there exist a flow-invariant splitting of the tangent bundle $T(SM) = \R F \oplus E_s \oplus E_u$ and constants $C \geq 1$ and $0 < \rho < 1$ such that, for all $v\in SM$ and $t \geq 0$, we have 
\begin{equation*}
\Vert d\varphi_{t}|_{E_s(v)}\Vert \leq C\rho^t, \qquad \Vert d\varphi_{-t}|_{E_u(v)}\Vert \leq C\rho^t,
\end{equation*}
where $\Vert \cdot\Vert$ denotes the operator norm induced by $g$. On the other hand, the flow is said to be \textit{projectively Anosov} (or that it admits a \textit{dominated splitting}) if there exist a flow-invariant splitting of the quotient tangent bundle $T(SM)/\mathbb{R} F = \mathcal{E}_s \oplus \mathcal{E}_u$ and constants $C \geq 1$ and $0 < \rho < 1$ such that, for all $v \in SM$ and $t \geq 0$, we have
\[ \Vert d\varphi_t|_{\mathcal{E}_s(v)}\Vert  \Vert d\varphi_{-t}|_{\mathcal{E}_u(\varphi_t(v))}\Vert  \leq C \rho^t.\]
While every Anosov flow is projectively Anosov, the converse is not true. The distinction does not appear in the geodesic or magnetic settings because volume-preserving projectively Anosov flows are, in fact, Anosov \cite[Proposition 2.34]{araujo10}. Due to the dissipative nature of thermostats, the projectively Anosov property is often the more natural condition to study (see \cite{mettler19, echevarria-cuesta25b}). 



Let $B(M, g)$ be the set of $\lambda\in \mathcal{C}^\infty(SM, \R)$ such that the thermostat $(M, g, \lambda)$ has no conjugate points, $A(M, g)$ the set of $\lambda$ for which it is Anosov, and $D(M, g)$ the set of $\lambda$ for which it is projectively Anosov. It is an interesting question to see how these sets relate to one another. By definition, we have $A(M,g)\subseteq D(M,g)$. 
Moreover, as previously mentioned, it is known that $A(M, g)\subseteq B(M,g)$. However, whether the inclusion $D(M,g)\subseteq B(M,g)$ also holds is still an open question \cite[Section 1.3]{echevarria-cuesta25b}. As a particular case, it is still unknown whether $D(\mathbb{S}^2,g) \neq \emptyset$; since $B(\mathbb{S}^2,g) = \emptyset$, this would give an example where the inclusion does not hold. On the other hand, it is shown in \cite[Lemma 4.2]{echevarria-cuesta25b} that the inclusion $A(M,g) \subseteq D(M,g)$ can be proper by showing that $D(\mathbb{T}^2,g) \neq \emptyset$ for any Riemannian metric $g$ on the 2-torus. 

Expanding on this analysis, we will show in Proposition \ref{proposition:always-anosov}  that $A(M,g) \neq \emptyset$ when $M$ has genus $\geq 2$, even if the underlying metric $g$ admits a pair of conjugate points. This should be contrasted with the examples in \cite[Section 7]{burns02}, where much more effort is needed in order to construct an example of an Anosov magnetic system whose underlying Riemannian metric admits a pair of conjugate points. This highlights the flexibility of thermostats beyond the volume-preserving setting.



Building on this theme, our first main result shows how the interior of $B(M,g)$ compares to $D(M,g)$. Using the fact that Anosov systems are $\mathcal{C}^1$ open (thanks to the Alekseev cone field criterion \cite[Proposition 5.1.7]{fisher19}), it follows that the set $A(M,g)$ is always open in the $\mathcal{C}^k$-topology for every $k\geq 1$. A similar argument (see \cite[Proposition 2.3]{sambarino16} for diffeomorphisms) shows that the same holds for $D(M,g)$. On the other hand, we will show in Proposition \ref{proposition:b-is-closed} that $B(M,g)$ is closed in each $\mathcal{C}^k$-topology for $k \geq 2$. With this in mind, we prove the following.

\begin{theorem}\label{theorem:c2-interior}
    The interior of $B(M, g)$ in the $\mathcal{C}^2$-topology is contained in $D(M,g)$. 
\end{theorem}

This theorem is an analogue of a result previously established for geodesic flows in \cite[Theorem A]{ruggiero91} and for Hamiltonian flows in \cite[Theorem 1]{contreras98}, although the lack of volume preservation means that we have to replace the Anosov property with the projectively Anosov property.
The main tools in the proof are the Green bundles in the cotangent bundle $T^*(SM)$, \cite[Theorem 1.6]{echevarria-cuesta25b} describing how $B(M,g)\cap D(M,g)$ sits inside $B(M,g)$, and a generalization of the index form to the thermostat setting.


Our next main result relates to the study of a potential inclusion $D(M,g)\subseteq B(M,g)$. In particular, we make some progress in the case where the thermostat is \textit{reversible}, in the sense that the flip $(x,v) \mapsto (x, -v)$ on $SM$ conjugates $\varphi_t$ with $\varphi_{-t}$. As we will show in Lemma \ref{lem:reversibility-odd-fourier}, these are the thermostats with odd $\lambda$, so they include Gaussian thermostats as well as projective flows (that is, geodesic flows of torsion-free affine connections), amongst other dynamical systems. 

Recurrence of the underlying system is key for our argument. Recall that a point $v\in SM$ is \textit{non-wandering} for the flow $\{\varphi_t\}_{t\in \R}$ if, for every neighbourhood $U$ of $v$ and $T >0$, there exists $t\geq T$ such that $\varphi_t(U)\cap U\neq \emptyset$. We denote by $\Omega\subseteq SM$ the set of all non-wandering points. The set $\Omega$ is non-empty, closed and flow-invariant. For Anosov thermostats, we have $\Omega=SM$ because the flow is topologically transitive by \cite[Theorem A]{ghys84}. However, the explicit examples in \cite[Section 4]{echevarria-cuesta25b} show that this property does not extend to the setting of projectively Anosov thermostats.




Building on the work in \cite{paternain94}, which takes advantage of a technique by Mañé \cite{mane87}, we generalize Klingenberg's result \cite{klingenberg74} from the setting of Anosov geodesic flows to reversible projectively Anosov thermostats. 

\begin{theorem}\label{theorem:dominated-splitting-theorem}
        The non-wandering set of a reversible projectively Anosov thermostat contains no conjugate points.
\end{theorem}

\begin{remark}
The non-wandering set $\Omega$ is a flow-invariant subset of $SM$, not $M$, so when we say it contains no conjugate points we mean that, for any $v \in \Omega$, there are no conjugate points along the thermostat geodesic $t \mapsto \pi(\varphi_t(v))$.
\end{remark}

The proof crucially depends on the reversibility of the flow, and we provide an example illustrating why the argument fails for non-reversible thermostats. Note that this does not imply that the result is false in the non-reversible case, but rather that a different proof would be required. 



\subsection*{Acknowledgements} We thank Gabriel Paternain for highlighting the importance of the reversibility assumption in the proof of Theorem \ref{theorem:dominated-splitting-theorem}. The first author was supported by a Harding Distinguished Postgraduate Scholarship, and the second author was supported by NSF DMS-2247747.

\section{Interior of thermostats without conjugate points}

\subsection{Preliminaries}

We briefly recall the set-up from \cite{echevarria-cuesta25b}. As previously, let $X$ be the geodesic vector field on $SM$ and let $V$ be the vertical vector field generating the circle action on the fibres. Define $H\coloneq[V,X]$. The vector fields $(X, H, V)$ form an orthonormal frame for $T(SM)$ with respect to the Sasaki metric, the natural lift of $g$ to $SM$. Dual to this moving frame is a moving frame for $T^*(SM)$, which we denote by $(\alpha, \beta, \psi)$. The \textit{cohorizontal subbundle} is defined as $\mathbb{H}^* \coloneq \mathbb{R} \beta$, whereas the \textit{covertical subbundle} is $\mathbb{V}^*\coloneq \mathbb{R}\psi$.


The symplectic lift of $\{\varphi_t\}_{t\in \R}$ to $T^*(SM)$, given by 
\begin{equation}
    \tilde{\varphi}_t(v, \xi) \coloneqq\left(\varphi_t(v), d_v\varphi_t^{-\top} (\xi)\right), \qquad (v,\xi)\in T^*(SM),
\end{equation}
is the Hamiltonian flow of $\xi(F(v))$, so it preserves the \textit{characteristic set} $\Sigma$ with fibres
$$\Sigma(v) \coloneq \{\xi \in T_v^*(SM) \mid \xi(F(v)) = 0\}.$$
Note that usually the characteristic set of $F$ is defined without the zero section. Also, we use the notation $d\varphi_t^{-\top}$ to denote the inverse of the transpose
$$d_v\varphi_t^{\top} : T_v^*(SM) \rightarrow T_{\varphi_{-t}(v)}^*(SM), \qquad d_v\varphi_t^{\top}\xi \coloneq \xi \circ d_{\varphi_{-t}(v)}\varphi_{t}.$$

In the geodesic case, we have $\Sigma = \mathbb{H}^*\oplus \mathbb{V}^*$. For thermostats, we introduce the $1$-form $\psi_\lambda : = \psi -\lambda\alpha$ and the \textit{tilted covertical subbundle} $\mathbb{V}^*_\lambda \coloneq \R \psi_\lambda$ so that $\Sigma = \mathbb{H}^* \oplus \mathbb{V}^*_\lambda$.  The adapted coframe $(\alpha, \beta, \psi_\lambda)$ provides convenient coordinates on $\Sigma$. Indeed, for each $(v,\xi) \in \Sigma$, we have
\begin{equation}\label{eq:equations-of-motion}
    d_v\varphi_t^{-\top}(\xi)= x(t)\beta+y(t)\phi_p
\end{equation}
for some unique functions $x,y\in \mathcal{C}^\infty(\R)$. These functions capture all the information of the lifted dynamics in $\Sigma$. Letting $K_g$ be the Gaussian curvature associated to $g$, \cite[Lemma 2.1]{echevarria-cuesta25b} shows that the function $y$ satisfies the Jacobi equation
\begin{equation} \label{eqn:jacobi} \ddot{y} + V(\lambda)\dot{y} + \mathbb{K} y = 0,\end{equation}
where $\mathbb{K}\in \mathcal{C}^\infty(SM, \R)$ is the \textit{thermostat curvature} given by $$\mathbb{K} \coloneq \pi^* K_g - H\lambda + \lambda^2 + FV\lambda.$$
Furthermore, the thermostat $(M,g,\lambda)$ is without conjugate points if and only if every solution $y$ to the Jacobi equation \eqref{eqn:jacobi} vanishes at most once \cite[Corollary 2.6]{echevarria-cuesta25b}. If we consider the \textit{damping function}
\[ m(t) \coloneq \exp \left( - \frac{1}{2} \int_0^t V(\lambda)(\varphi_\tau(v)) d\tau \right),\]
then we can associate to each solution $y$ its \textit{damped} version $z(t) \coloneq y(t)/m(t)$. By \cite[Lemma 2.3]{echevarria-cuesta25b}, this function satisfies the Jacobi equation
\begin{equation} \label{eqn:jacobi-damped} \ddot{z} + \tilde{\kappa} z= 0,\end{equation}
where $\tilde{\kappa}\in \mathcal{C}^\infty(SM, \R)$ is the \textit{damped thermostat curvature} given by 
\begin{equation}\label{eq:definition-damped-curvature}
    \tilde{\kappa} \coloneq \mathbb{K} - FV\lambda/2 - (V\lambda)^2/4.
\end{equation}
These solutions still encode information about conjugate points. Since the damping function $m$ is nowhere-vanishing, a thermostat is without conjugate points if and only if every non-trivial solution to equation \eqref{eqn:jacobi-damped} vanishes at most once.

Let us quickly justify why the sets of interest are not empty when $M$ has genus $\geq 1$, even if the metric $g$ is not necessarily Anosov.

\begin{proposition}\label{proposition:always-anosov}
    For any closed oriented Riemannian surface $(M,g)$ of genus $\geq 2$, the set $A(M, g)$ is non-empty.
\end{proposition}

\begin{proof}
    Define $f\coloneq K_g-2\pi \chi(M)$ and let $\mu_a$ be the area form on $M$. The Gauss-Bonnet theorem implies that
    $$\int_M f\mu_a = 0.$$
    Therefore, we can apply the de Rham theorem to write $f\mu_a = d\theta$ for some 1-form $\theta$ on $M$. Let $E$ be the vector field on $M$ characterized by $\iota_E \mu_a=\theta$. 
    Then, 
    \begin{equation*}
        \mathcal{L}_E\mu_a = d\iota_E\mu_a =d\theta=f\mu_a, 
    \end{equation*}
    so $\div_{\mu_a}E=f$. If we define $\lambda\in \mathcal{C}^\infty(SM, \R)$ by $$\lambda(x,v)\coloneq\theta_x(v)=g_x(E(x),J(v)),$$ then the thermostat curvature of the Gaussian thermostat $(M, g, \lambda)$ satisfies
    $$\mathbb{K}=\pi^*(K_g -\div_{\mu_a}E)=\pi^*(K_g-f)=2\pi\chi(M) <0.$$
    By \cite[Theorem 5.2]{wojtkowski00}, it follows that the thermostat is Anosov.
\end{proof}

\begin{remark}
    The proof shows that $\lambda$ can actually be chosen so that $(M, g, \lambda)$ is a Gaussian thermostat. In particular, this example is reversible. 
\end{remark}

Combining this with \cite[Lemma 4.2]{echevarria-cuesta25b}, we get the following:

\begin{corollary}
    For any closed oriented Riemannian surface $(M,g)$ of genus $\geq 1$, the set $B(M, g)\cap D(M,g)$ is non-empty.
\end{corollary}


\subsection{The index form}

We now want to introduce a thermostat index form that parallels the classical index form from Riemannian geometry. Let $P\mathcal{C}^2$ denote the set of piecewise $\mathcal{C}^2$ functions. For a given $v \in SM$, we define the \textit{thermostat index form} $I_v : P \mathcal{C}^2([-T,T], \R) \rightarrow \mathbb{R}$ by
\begin{equation} \label{eqn:damped-index} I_v(f) \coloneq \int_{-T}^T \left( (\dot{f}(t))^2 - \tilde{\kappa}_v(t) (f(t))^2\right) dt, 
\end{equation}
where the parameter $T > 0$ is understood from context and $\tilde{\kappa}_v(t) \coloneq \tilde{\kappa}(\varphi_t(v))$. Our first step toward proving Theorem \ref{theorem:c2-interior} is to show that this index form continues to detect the presence of conjugate points. The proof of the following is a standard exercise in Riemannian geometry (for example, see \cite[Lemma 5.7]{herreros10}), but we sketch it here for the reader's convenience.

\begin{proposition} \label{propn:index}
    Let $(M,g,\lambda)$ be a thermostat, and let $v \in SM$ and $T > 0$ be such that the curve $t \mapsto \pi(\varphi_t(v))$ has no conjugate points for $t \in [-T,T]$. Let $z$ be a solution to the Jacobi equation \eqref{eqn:jacobi-damped} and let $f \in P \mathcal{C}^2([-T,T], \R)$. If $z(-T) = f(-T) = 0$ and $z(T) = f(T)$, then $I_v(z) \leq I_v(f)$, with equality if and only if $f = z$.
\end{proposition}

\begin{proof}[Sketch of proof]
    We prove the result assuming that $f \in \mathcal{C}^2([-T,T], \R)$. The proof for the piecewise setting is essentially the same, but the notation becomes cumbersome. Since the curve is without conjugate points, let $w$ be the unique solution to the Jacobi equation \eqref{eqn:jacobi-damped} with $w(-T) = 0$ and $w(T) = 1$, and let $q \in \mathcal{C}^2([-T,T], \R)$ be such that $qw = f$. Using equation \eqref{eqn:jacobi-damped}, we have
    \begin{equation} \label{eqn:lem-index-step1} \dot{f} = \dot{q}w + q \dot{w} \qquad \text{and}\qquad \ddot{f} = \ddot{q} w + 2 \dot{q} \dot{w} - \tilde{\kappa}f.\end{equation}
    Also note that $z = q(T)w$, so $z(T) = q(T)$ and $\dot{z}(T) = q(T) \dot{w}(T)$. Using equation \eqref{eqn:jacobi-damped} again, we get
    \begin{equation} \label{eqn:lem-index-step2} 
    I_v(z) = z(T) \dot{z}(T) = (q(T))^2 \dot{w}(T) = f(T) q(T) \dot{w}(T).
    \end{equation}
    Using integration by parts, we obtain
    \[ \begin{split} I_v(f) &= \int_{-T}^T \left( (\dot{f}(t))^2 - \tilde{\kappa}_v(t) (f(t))^2\right) dt \\
    &= f(T) \dot{f}(T) -f(-T) \dot{f}(-T) -  \int_{-T}^{T}  f(t)\left(\ddot{f}(t) +   \tilde{\kappa}_v(t) f(t)\right)dt \\
    & = f(T) \dot{q}(T) w(T) + f(T) q(T) \dot{w}(T) - f(-T) \dot{q}(-T) w(-T) \\
    & \quad \quad - f(-T) q(-T) \dot{w}(-T) - \int_{-T}^{T}  f(t)\left(\ddot{f}(t) +   \tilde{\kappa}_v(t) f(t)\right)dt.\end{split}\]
    With equations \eqref{eqn:lem-index-step1}, we can rewrite the integral on the right-hand side as
    \[\begin{split}   \int_{-T}^{T}  f(t)\left(\ddot{f}(t) +   \tilde{\kappa}_v(t) f(t)\right) dt &=   \int_{-T}^{T}  f(t)\left(\ddot{q}(t)w(t) + 2 \dot{q}(t)\dot{w}(t)\right)dt \\
    & = \int_{-T}^{T}  f(t)\left( \frac{d}{dt}(\dot{q}(t)w(t)) +  \dot{q}(t)\dot{w}(t)\right)dt.\end{split}\]
    Integration by parts again yields
    \[ \begin{split}   \int_{-T}^{T}  f(t)\left(\ddot{f}(t) +   \tilde{\kappa}_v(t) f(t)\right) dt = &f(T) \dot{q}(T) w(T) - f(-T)\dot{q}(-T) w(-T) \\ & +\int_{-T}^{T} \dot{q}(t)\left(  f(t) \dot{w}(t)-\dot{f}(t) w(t) \right) dt. \end{split}\] 
    Putting the above together with equations \eqref{eqn:lem-index-step2}, we obtain
    \[ I_v(f) = I_v(z) + \int_{-T}^{T} \dot{q}(t)\left(\dot{f}(t) w(t) - f(t)  \dot{w}(t) \right) dt.\]
    Now, use equations \eqref{eqn:lem-index-step1} again to get that 
    \[ \dot{f} \dot{q} w =  \left( \dot{q} w\right)^2 + q \dot{q} w \dot{w} \qquad \text{and} \qquad f \dot{q} \dot{w} = q \dot{q} w \dot{w}, \]
    and hence
    $$I_v(f) = I_v(z) + \int_{-T}^{T} (\dot{q}(t) w(t))^2  dt.$$
    Thus, we have $I_v(f) \geq I_v(z)$. Note that equality holds if and only if
    \[ \int_{-T}^{T} (\dot{q}(t)w(t))^2 dt = 0.\]
    Since $w(t) \neq 0$ for all $t \in (-T, T]$ by the no-conjugate-points assumption, we deduce that $\dot{q}= 0$. This implies that $q(t) = q(T)$ for all $t \in [-T,T]$ and hence $f = z$.
\end{proof}

A consequence of Proposition \ref{propn:index} is the following criteria for a thermostat to be without conjugate points.

\begin{lemma} \label{lem:index}
    A thermostat $(M, g, \lambda)$ has no conjugate points if and only if for every $v \in SM$, $T > 0$ and $f \in P\mathcal{C}^2([-T,T], \R)$ with $f(-T) = f(T) = 0$, we have $I_v(f) \geq 0$, with equality if and only if $f$ is zero.
\end{lemma}

\begin{proof}
If $(M, g, \lambda)$ has no conjugate points, then the only solution to the Jacobi equation \eqref{eqn:jacobi-damped} with ${z(-T) = z(T) = 0}$ is the trivial solution, so the claim follows from Proposition \ref{propn:index}. For the other direction, we prove the contrapositive. Suppose $(M, g, \lambda)$ admits a pair of conjugate points. By \cite[Corollary 2.6]{echevarria-cuesta25b}, there exists $v \in SM$ and $T > 0$ such that there is a non-trivial solution $z$ to the Jacobi equation \eqref{eqn:jacobi-damped} with $z(a) = z(b) = 0$ and $[a,b] \subseteq [-T,T]$. Let $f \in P\mathcal{C}^2([-T,T], \R)$ be such that $f = z$ on $[a,b]$ and $f=0$ otherwise. Integration by parts then yields
\[ I_v(f) = \int_{a}^b \left((\dot{z}(t))^2 - \tilde{\kappa}_v(t) (z(t))^2\right)dt = z(b) \dot{z}(b) - z(a) \dot{z}(a) = 0. \qedhere \]
\end{proof}

A slight modification to the above argument gives us the following useful criteria.

\begin{lemma} \label{lem:neg-index}
    If a thermostat $(M, g, \lambda)$ admits a pair of conjugate points, then there exist $v \in SM$, $T > 0$ and $f \in P\mathcal{C}^2([-T,T], \R)$ such that $I_v(f) < 0$. 
\end{lemma}

\begin{proof}
If $(M,g,\lambda)$ admits a pair of conjugate points and $f$ is the piecewise function from the proof of Lemma \ref{lem:index}, then define $q \in P \mathcal{C}^2([-T,T], \R)$ by $q(t) \coloneq  \dot{f}(t^-)$ and set $f_\varepsilon^\pm \coloneq f \pm \varepsilon q$ for some $\varepsilon>0$. Observe that the Jacobi equation \eqref{eqn:jacobi-damped} and the proof of Lemma \ref{lem:index} yield
\[ \begin{split}  I_v(f_\varepsilon^\pm) &= \pm 2\varepsilon \int_a^b \left( \dot{z}(t) \ddot{z}(t) - \tilde{\kappa}_v(t) z(t) \dot{z}(t)  \right)  dt + O(\varepsilon^2) \\
& = \mp 4\varepsilon \int_a^b \tilde{\kappa}_v(t) z(t) \dot{z}(t) dt + O(\varepsilon^2).\end{split}\]
After appropriately choosing the sign, we have $I_v(f_\varepsilon^\pm) < 0$ for $\varepsilon > 0$ small enough
\end{proof}

As a consequence, we can deduce the following observation.

\begin{proposition}\label{proposition:b-is-closed}
    The set $B(M,g)$ is closed in the $\mathcal{C}^k$-topology for all $k \geq 2$.
\end{proposition}

\begin{proof}
    If $\{\lambda_n\}_{n\in \N} \subset B(M,g)$ is a sequence of functions such that $\lambda_n$ converges to $\lambda$ in the $\mathcal{C}^\infty$-topology, then the corresponding damped thermostat curvatures for each $\lambda_n$ converge to the damped thermostat curvature of $\lambda$ in the $\mathcal{C}^\infty$-topology, and so, for any fixed $v \in SM$ and $T > 0$ the corresponding thermostat index forms converge. Denoting the index form for $(M, g,\lambda)$ by $I$, we see that $I_v(f) \geq 0$ for every $f \in P \mathcal{C}^2([-T,T], \R)$, and the contrapositive of Lemma \ref{lem:neg-index} shows that $\lambda \in B(M,g)$.
\end{proof}




\subsection{Green bundles in the cotangent space}
The usual construction of the Green bundles in $T(SM)$ does not carry over to the thermostat setting due to dissipation. We instead define them in $T^*(SM)$, following \cite{echevarria-cuesta25b}. In \cite[Theorem 1.5]{echevarria-cuesta25b}, we showed that if a thermostat $(M,g,\lambda)$ has no conjugate points, then the subbundle $d\varphi_{t}^{-\top}(\mathbb{H}^*(\varphi_t(v))$ converges as $t\to \pm \infty$. The subbundle that it converges to in forward time is the \textit{stable Green bundle}, denoted by $G_s^*(v)$, and the bundle it converges to in reverse time is the \emph{unstable Green bundle}, denoted by $G_u^*(v)$. When we talk about convergence, we mean with respect to the Grassmannian topology of the projective bundle $\mathbb{P}(\Sigma)$ defined in Subsection \ref{subsection:maslov-index}.

Let $z_T$ be a solution to the Jacobi equation \eqref{eqn:jacobi-damped} satisfying
\begin{equation} \label{eqn:defn-zt}
z_T(T) = 0 \quad \text{ and } \quad z_T(0) = 1.
\end{equation}
The following is a consequence of \cite[Theorem 1.5, Lemma 2.4]{echevarria-cuesta25b} (see also \cite{green54}).

\begin{lemma}\label{lem:convergence-zt}
    Let $(M,g,\lambda)$ be a thermostat without conjugate points. The function $z_T$ is unbounded and $\dot{z}_T(0)$ converges as $T\to \pm \infty$.
\end{lemma}


\begin{remark} \label{rem:green-bundles}
    If $z_s$ is the solution to the Jacobi equation \eqref{eqn:jacobi-damped} with initial conditions $$(z_s(0), \dot{z}_s(0)) = (1, \lim_{T \rightarrow \infty} \dot{z}_T(0)),$$ then the corresponding vector $\xi_s \in \Sigma(v)$ spans the stable Green bundle. The same is true for $z_u$ with $(z_s(0), \dot{z}_s(0)) = (1, \lim_{T \rightarrow \infty} \dot{z}_T(0))$ and the unstable Green bundle.
\end{remark}

Note that the Green bundles are flow-invariant, and key for our proof is the fact that $\lambda \in D(M,g)$ if and only if the Green bundles of $(M,g,\lambda)$ satisfy $G_s^*(v) \cap G_u^*(v) = \{0\}$ for every $v \in SM$ \cite[Theorem 1.6]{echevarria-cuesta25b}.


\subsection{Proof of Theorem \ref{theorem:c2-interior}}

We now fix a Riemannian metric $g$ on $M$. As in the introduction, let $B(M,g) \subset \mathcal{C}^\infty(SM, \R)$ be the set of functions $\lambda$ such that $(M,g,\lambda)$ has no conjugate points, and let $D(M,g)$ be the set of functions such that $(M,g,\lambda)$ is projectively Anosov. The goal is to take $\lambda \in B(M,g) \setminus D(M,g)$ and show that, for every $\varepsilon > 0$, there exists $\phi_\varepsilon \in \mathcal{C}^\infty(SM, \R)$ such that $\|\phi_\varepsilon\|_{\mathcal{C}^2(SM)} \leq  \varepsilon$ and $\lambda + \phi_\varepsilon \in B(M,g)^c.$

\begin{remark}
    In absence of a thermostat version of Klingenberg's result, it is possible that $\lambda + \phi_\varepsilon \in D(M,g)$ (see \cite[Remark 1.7]{echevarria-cuesta25b}). 
\end{remark}
To begin, fix $\varepsilon > 0$. 
By \cite[Theorem 1.6]{echevarria-cuesta25b}, there exists a point $v \in SM$ such that the Green bundles are equal along the orbit $t \mapsto \varphi_t(v)$. If this orbit is closed, let $C>0$ be half its period; otherwise, choose $C > 0$ arbitrarily. For each $T > 0$, let $z_{ T}$ be the solution to the Jacobi equation \eqref{eqn:jacobi-damped} satisfying the conditions \eqref{eqn:defn-zt}. Let $f_T \in P \mathcal{C}^2([-T,T], \R)$ be the function defined by 
\[ f_T(t) \coloneq \begin{cases} z_{-T}(t) & t \in [-T,0],\\ z_T(t) & t \in [0,T]. \end{cases}\]

\begin{figure}[h]
    \centering
    \includegraphics[scale=0.5]{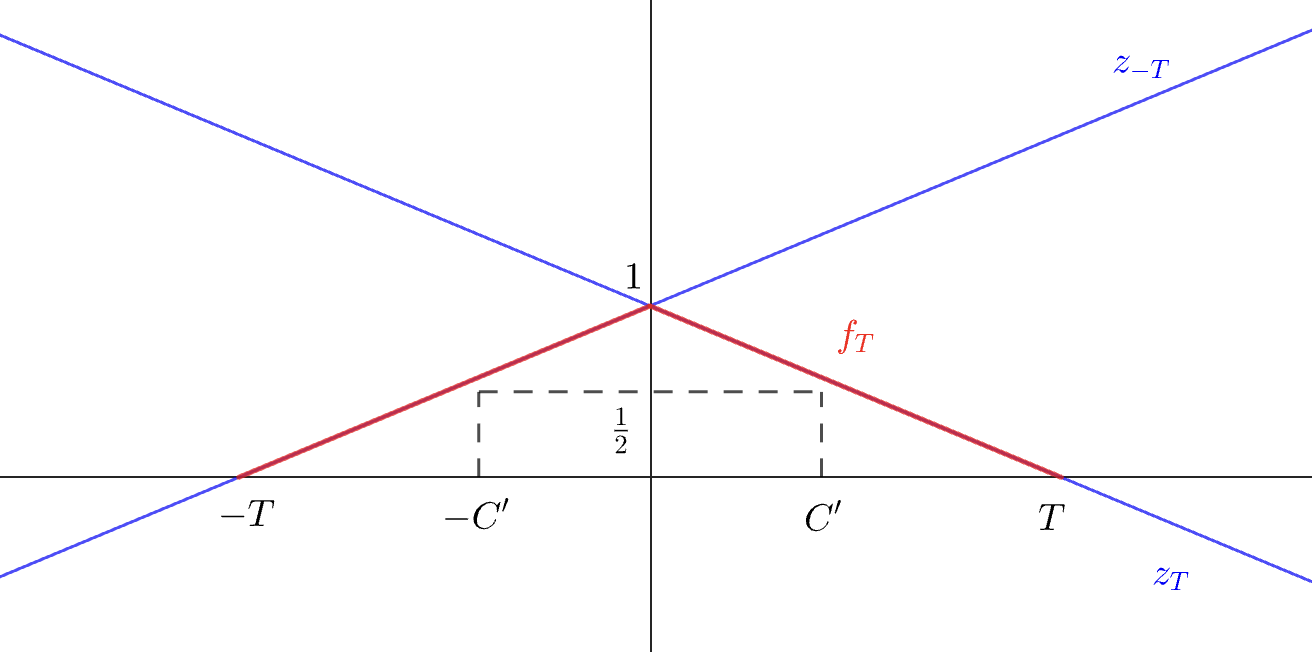}
    \caption{A schematic plot of the functions $z_{-T}$, $z_T$ and $f_T$.}
    \label{fig:enter-label}
\end{figure}

See Figure \ref{fig:enter-label} for an illustration of the construction of $f_T$. Using Lemma \ref{lem:convergence-zt} and Remark \ref{rem:green-bundles}, we see that, for $T$ large enough, there exists a constant $C' \in (0,C/2)$ such that $f_T(t) > 1/2$ for $|t| \leq C'$. Using integration by parts, we deduce that 
\[\begin{split} I_v(f_T) &= \int_{-T}^0 \left(\dot{z}_{-T}(t)^2 - \tilde{\kappa}_v(t) z_{-T}(t)^2 \right) dt +  \int_{0}^T \left(\dot{z}_{T}(t)^2 - \tilde{\kappa}_v(t) z_{T}(t)^2 \right) dt  \\
& = \dot{z}_{-T}(0) - \dot{z}_T(0).\end{split}\]
Since the Green bundles coincide along this orbit, Remark \ref{rem:green-bundles} implies that we have $|I_v(f_T)| = |\dot{z}_{-T}(0) - \dot{z}_T(0)| \rightarrow 0$ as $T \rightarrow \infty$. In particular, we may assume that $T$ is sufficiently large so that 
\begin{equation} \label{eqn:estimate1}
    I_v(f_T) \leq \varepsilon C'/4.
\end{equation} 
If the orbit is periodic, then by making $T$ larger if needed we may also assume that there is a $k \in \mathbb{N}$ such that $kC = T$.

With this estimate out of the way, the next step is to construct the perturbing function $\phi_\varepsilon$. Let $\chi \in \mathcal{C}^\infty_c(\mathbb{R}, \R)$ be a bump function with support in $[-1,1]$ which is $1$ on $[-1/2, 1/2]$. Consider the flow box coordinates $(x_1,x_2,x_3)$ around the orbit segment $\{\varphi_t(v) \mid t \in [-C,C]\}$ such that $\partial_{x_1} = F$, $\partial_{x_2} = H$ and $\partial/\partial x_3 = V$, and define the function
\begin{equation} \label{eqn:perturb-defn}
    \phi_\varepsilon(x_1, x_2, x_3) \coloneq \sqrt{\varepsilon} \chi \left( \frac{x_1}{C}\right) \chi\left( \frac{\sqrt{x_2^2 + x_3^2}}{\delta} \right),
\end{equation} 
where $\delta > 0$ is to be determined; see Figure \ref{fig:function-phi} for an illustration of $\varphi_\varepsilon$. 

\begin{figure}[h]
    \centering
    \includegraphics[scale=0.4]{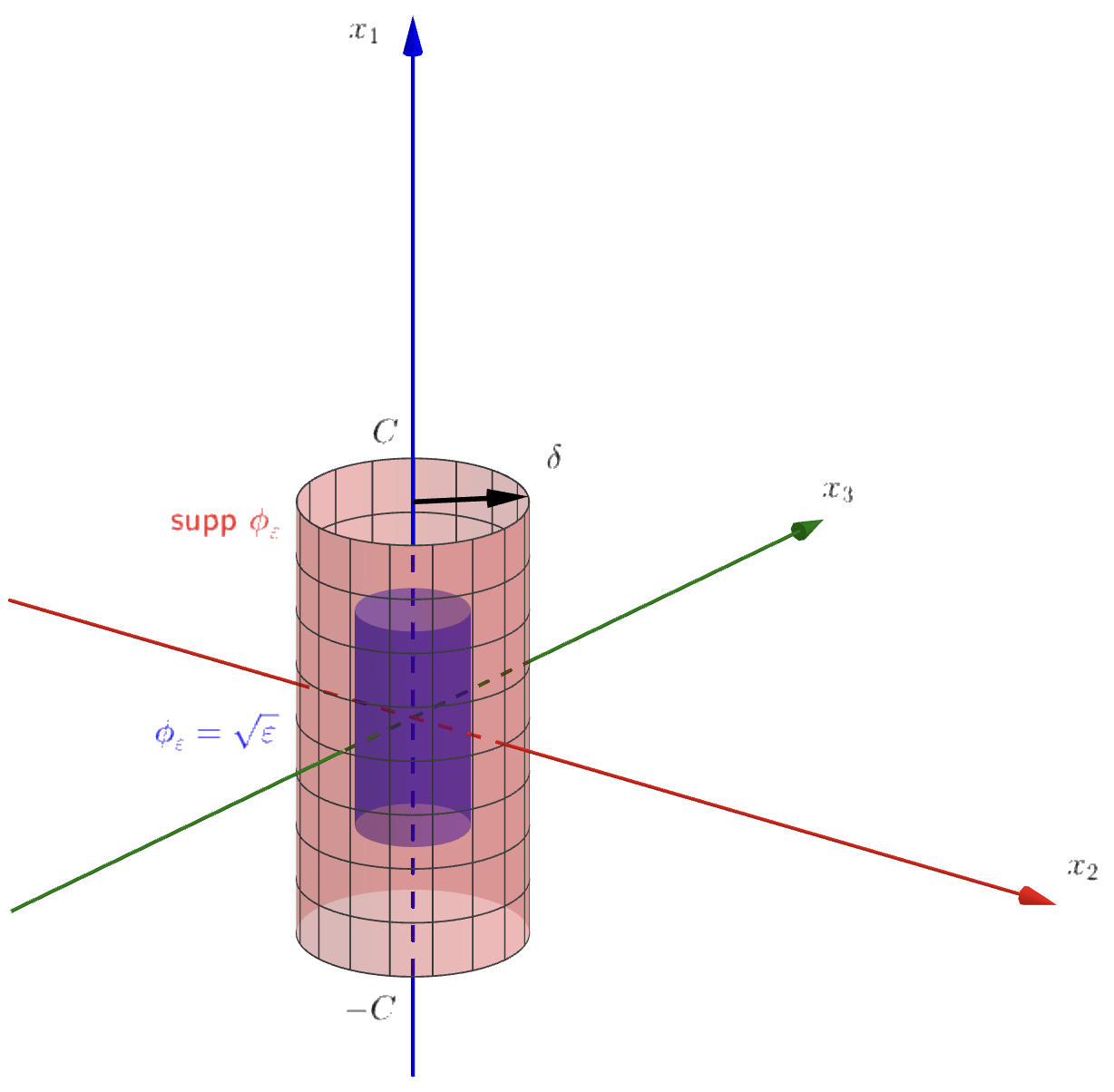}
    \caption{Illustration of the function $\phi_\varepsilon$.}
    \label{fig:function-phi}
\end{figure}

The aim for the rest of this section is to prove the following lemma.

\begin{lemma} \label{propn:c2}
    For every $\varepsilon > 0$, there exists an appropriate choice of $\delta > 0$ and sign $\pm$ such that the thermostat index form $\bar{I}^\varepsilon$ for $(M,g,\lambda \pm \phi_\varepsilon)$ satisfies $\bar{I}^\varepsilon_v(f_T) < 0$ for the given $v \in SM$ and $T > 0$.
\end{lemma}

Assuming the lemma holds, we use the contrapositive of Lemma \ref{lem:index} to deduce that $(M,g,\lambda \pm \phi_\varepsilon)$ must admit a pair of conjugate points. Since $\|\phi_\varepsilon\|_{\mathcal{C}^2(SM)} \rightarrow 0$ as $\varepsilon \rightarrow 0^+$, Theorem \ref{theorem:c2-interior} follows.

\begin{proof}[Proof of Lemma \ref{propn:c2}]
    Let $\tilde{\kappa}_\varepsilon$ be the damped thermostat curvature of $(M,g,\lambda \pm \phi_\varepsilon)$ and let $\{\varphi^\varepsilon_t\}_{t\in \R}$ be the thermostat flow. By adding and subtracting terms from the equation \eqref{eqn:damped-index}, we get
    \[ \begin{split} I^\varepsilon_v(f_T) = I_v(f_T) &+ \int_{-T}^T \left(\tilde{\kappa}(\varphi_t(v)) - \tilde{\kappa}_\varepsilon(\varphi_t(v))\right) (f_T(t))^2\, dt \\
    & + \int_{-T}^T \left(\tilde{\kappa}_\varepsilon(\varphi_t(v)) - \tilde{\kappa}_\varepsilon(\varphi^\varepsilon_t(v)) \right) (f_T(t))^2\, dt.
    \end{split}\]
    Observe that the support of $\phi_\varepsilon$ gets uniformly smaller as $\delta \rightarrow 0$. On the interval $[-T,T]$, we may choose $\delta$ sufficiently small such that $\varphi^\varepsilon_t(v)$ is uniformly close to $\varphi_t(v)$ for $t \in [-T,T]$. Using inequality \eqref{eqn:estimate1}, we may assume that $\delta$ is small enough so that
    \begin{equation} \label{eqn:estimate2} I^\varepsilon_v(f_T) \leq \frac{\varepsilon C'}{2} + \int_{-T}^T \left(\tilde{\kappa}(\varphi_t(v)) - \tilde{\kappa}_\varepsilon(\varphi_t(v))\right)  (f_T(t))^2\, dt.\end{equation}
   
    The next step is to analyse the integral of the differences between the damped curvatures. Using definition \eqref{eq:definition-damped-curvature}, we can rewrite $\tilde{\kappa}_\varepsilon$ in terms of $\tilde{\kappa}$ as follows:
    \[ \tilde{\kappa}_\varepsilon = \tilde{\kappa} - H\phi_\varepsilon + 2\lambda \phi_\varepsilon + \phi_\varepsilon^2 + \frac{\phi_\varepsilon V^2\lambda}{2} + \frac{(X + \lambda_\varepsilon V)(V\phi_\varepsilon)}{2} - \frac{V(\lambda) V(\phi_\varepsilon)}{2} - \frac{(V\phi_\varepsilon)^2}{4}. \]
    On the orbit segment $\{\varphi_t(v) \mid t \in [-C,C]\}\subset SM$, we have $H\phi_\varepsilon=0$ and $V\phi_\varepsilon =0$, so we can express the differences in the damped curvature as
    \[ \tilde{\kappa} - \tilde{\kappa}_\varepsilon = - \phi_\varepsilon^2 -  \left(2 \lambda - \frac{V^2\lambda}{2} \right) \phi_\varepsilon.\]
    After potentially adjusting the sign of $\phi_\varepsilon$, we can shrink $\delta$ so that the following holds:
    \begin{equation} \label{eqn:estimate3} \int_{-C}^C \left(\tilde{\kappa}(\varphi_t(v)) - \tilde{\kappa}_\varepsilon(\varphi_t(v))\right) (f_T(t))^2\, dt \leq -\varepsilon C'. \end{equation}

    Let
    \[\Delta \coloneqq \{t \in [-T,T] \setminus [-C,C] \mid \varphi_t(v) \in \text{supp}(\phi_\varepsilon)\}.\] 
    Note that $\tilde{\kappa} = \tilde{\kappa}_\varepsilon$ outside of the support of $\phi_\varepsilon$, and so using inequality \eqref{eqn:estimate3} we are left with
    \[ \int_{-T}^T \left(\tilde{\kappa}(\varphi_t(v)) - \tilde{\kappa}_\varepsilon(\varphi_t(v))\right) (f_T(t))^2\,  dt \leq -\varepsilon C' + \int_{\Delta}\left(\tilde{\kappa}(\varphi_t(v)) - \tilde{\kappa}_\varepsilon(\varphi_t(v))\right) (f_T(t))^2\,  dt. \]
    If the orbit is not periodic, then by compactness there is a non-zero distance between the orbit segments coming from $\Delta$ and the orbit segment $\{\varphi_t(v) \mid t \in [-C,C]\}$. In particular, we can assume that $\delta$ is small enough so that $\Delta = \emptyset$. If the orbit is periodic, then the same logic implies that for $\delta$ small enough we have
    \begin{equation*}
    \begin{alignedat}{1}
     \int_{\Delta}\left(\tilde{\kappa}(\varphi_t(v)) - \tilde{\kappa}_\varepsilon(\varphi_t(v))\right) (f_T(t))^2\,  dt  &= (k-1) \int_{-C}^C \left(\tilde{\kappa}(\varphi_t(v)) - \tilde{\kappa}_\varepsilon(\varphi_t(v))\right) (f_T(t))^2\,  dt \\
    &\leq -(k-1) \varepsilon C',
    \end{alignedat}
    \end{equation*}
    since we arranged to have $kC = T$ for some integer $k\geq 1$. In either case, there is a $\delta$ small enough so that we can rewrite inequality \eqref{eqn:estimate2} as
    \[ I^\varepsilon_v(f_T) \leq \frac{\varepsilon C'}{2} - \varepsilon C' < 0. \qedhere\]
\end{proof}

\begin{remark}
    This perturbation is easier in the thermostat case than if we wanted to keep it, say, as a pure magnetic perturbation.
\end{remark}

\section{Reversible thermostats}

\subsection{Mirror thermostats}
The \textit{flip map} $\mathcal{F} : TM \rightarrow TM$ is given by $(x,v)\mapsto (x,-v)$. We define the \textit{mirror thermostat} of $(M,g, \lambda)$ to be $(M,g,\lambda^{\mathcal{F}})$, where $\lambda^{\mathcal{F}} \coloneq - \lambda \circ \mathcal{F}$. The following lemma relates the flow of a thermostat and that of its mirrored version.

\begin{lemma}\label{lemma:conjugacy}
    If $\{\varphi_t\}_{t\in \R}$ is the flow associated to $(M,g,\lambda)$, then $\{\mathcal{F} \circ \varphi_{-t} \circ \mathcal{F}\}_{t\in \R}$ is the flow associated to $(M,g,\lambda^{\mathcal{F}})$.
\end{lemma}

\begin{proof}
For the sake of clarity, we use the notation $\{\varphi_t^\lambda\}_{t\in \R}$ to denote the flow on $SM$ with respect to $(M, g,\lambda)$. Let $\gamma$ be a thermostat geodesic for $(M, g,\lambda)$ and let $\eta(t) \coloneqq \gamma(-t)$. Then,
    \[\nabla_{\dot{\eta}(t)} \dot{\eta}(t) =  \nabla_{\dot{\gamma}(-t)} \dot{\gamma}(-t) = \lambda(\gamma(-t), \dot{\gamma}(-t)) J \dot{\gamma}(-t) = \lambda^{\mathcal{F}}(\eta(t), \dot{\eta}(t)) J \dot{\eta}(t), \]
    where we have used the fact that $\dot{\gamma}(-t) = - \dot{\eta}(t)$.
    We see that $\eta$ is a thermostat geodesic for $(M, g,\lambda^{\mathcal{F}})$. Now notice that 
    \[ \begin{split} \varphi_{-t}^{\lambda}(\gamma(0), \dot{\gamma}(0)) &= (\gamma(-t), \dot{\gamma}(-t)) \\
    &= (\eta(t), - \dot{\eta}(t)) \\
    &= \mathcal{F}(\eta(t), \dot{\eta}(t)) \\ &=( \mathcal{F} \circ \varphi_t^{\lambda^{\mathcal{F}}})(\eta(0), \dot{\eta}(0)) \\
    & = (\mathcal{F} \circ \varphi_t^{\lambda^{\mathcal{F}}} \circ \mathcal{F} )(\gamma(0),\dot{\gamma}(0)). \qedhere\end{split}\]
\end{proof}

By definition, the thermostat is reversible if $\mathcal{F}$ conjugates $\varphi_t$ with $\varphi_{-t}$. Therefore, it is reversible if and only if $\lambda^{\mathcal{F}} = \lambda$, that is, if the function $\lambda$ is odd. In the next lemma, we will show how this reversibility condition is encoded at the level of the Fourier series expansions in the angular variable, that is, with respect to the flow $\{\rho_t\}_{t\in \R}$ generated by the vertical vector field $V$. Following \cite[Section 2.1.4]{echevarria-cuesta25}, observe that the space $\mathcal{C}^\infty(SM)$ admits the decomposition
\[ \mathcal{C}^\infty(SM) = \bigoplus_{k \in \mathbb{Z}} \Omega_k, \quad \Omega_k \coloneq \{ u \in \mathcal{C}^\infty(SM) \mid Vu = iku\}.\]
In particular, we can write
$\lambda = \sum_{k \in \mathbb{Z}} \lambda_k,$ where the \textit{$k$th Fourier mode} $\lambda_k\in \Omega_k$ is
\begin{equation*} \label{eqn:fourier-mode} \lambda_k(x,v) \coloneq \frac{1}{2\pi} \int_0^{2\pi} \lambda(\rho_t(x,v)) e^{-ikt} dt.\end{equation*}
We notice that
  \[ \lambda^{\mathcal{F}}_k(x,v) = - \frac{1}{2\pi} \int_0^{2\pi} \lambda(\rho_{t+\pi}(x,v)) e^{-ikt}  dt = e^{i(k+1)\pi}\lambda_k(x,v),\]
  so the following holds in light of Lemma \ref{lemma:conjugacy}.

\begin{lemma}\label{lem:reversibility-odd-fourier}
    A thermostat $(M,g,\lambda)$ is reversible if and only if $\lambda$ has only odd Fourier modes.
\end{lemma}

We will also need the following definition. See Figure \ref{fig:mirror-geodesic}.


\begin{definition}
    Given a curve $c :\mathbb{S}^1\to SM$, $c(t)=(x(t),v(t))$, its \emph{mirrored curve} is the map $\overline{c}: \mathbb{S}^1\to SM$ given by $\overline{c}(t)\coloneqq(x(-t), -v(-t))$.
\end{definition}

Note that if $c$ is a closed orbit of $(M,g, \lambda)$, then $\overline{c}$ is a closed orbit of $(M, g, \lambda^\mathcal{F})$.

\begin{center}
\begin{figure}[h]
    \centering
    \tikzset{every picture/.style={line width=0.75pt}} 

\begin{tikzpicture}[x=0.75pt,y=0.75pt,yscale=-1,xscale=1]

\draw   (201.5,212) .. controls (201.5,200.13) and (211.32,190.5) .. (223.42,190.5) .. controls (235.53,190.5) and (245.34,200.13) .. (245.34,212) .. controls (245.34,223.87) and (235.53,233.5) .. (223.42,233.5) .. controls (211.32,233.5) and (201.5,223.87) .. (201.5,212) -- cycle ;
\draw   (178.34,166.5) .. controls (215.68,141.37) and (296.4,121) .. (358.64,121) .. controls (420.88,121) and (441.06,141.37) .. (403.72,166.5) .. controls (366.38,191.63) and (285.66,212) .. (223.42,212) .. controls (161.18,212) and (141,191.63) .. (178.34,166.5) -- cycle ;
\draw   (156.42,166.5) .. controls (156.42,154.63) and (166.23,145) .. (178.34,145) .. controls (190.44,145) and (200.25,154.63) .. (200.25,166.5) .. controls (200.25,178.37) and (190.44,188) .. (178.34,188) .. controls (166.23,188) and (156.42,178.37) .. (156.42,166.5) -- cycle ;
\draw   (336.72,121) .. controls (336.72,109.13) and (346.53,99.5) .. (358.64,99.5) .. controls (370.74,99.5) and (380.56,109.13) .. (380.56,121) .. controls (380.56,132.87) and (370.74,142.5) .. (358.64,142.5) .. controls (346.53,142.5) and (336.72,132.87) .. (336.72,121) -- cycle ;
\draw   (381.81,166.5) .. controls (381.81,154.63) and (391.62,145) .. (403.72,145) .. controls (415.83,145) and (425.64,154.63) .. (425.64,166.5) .. controls (425.64,178.37) and (415.83,188) .. (403.72,188) .. controls (391.62,188) and (381.81,178.37) .. (381.81,166.5) -- cycle ;
\draw  [color={rgb, 255:red, 255; green, 0; blue, 0 }  ,draw opacity=1 ][dash pattern={on 0.84pt off 2.51pt}] (178.34,145) .. controls (215.68,119.87) and (296.4,99.5) .. (358.64,99.5) .. controls (420.88,99.5) and (441.06,119.87) .. (403.72,145) .. controls (366.38,170.13) and (285.66,190.5) .. (223.42,190.5) .. controls (161.18,190.5) and (141,170.13) .. (178.34,145) -- cycle ;
\draw  [color={rgb, 255:red, 0; green, 15; blue, 255 }  ,draw opacity=1 ][dash pattern={on 0.84pt off 2.51pt}] (178.34,188) .. controls (215.68,162.87) and (296.4,142.5) .. (358.64,142.5) .. controls (420.88,142.5) and (441.06,162.87) .. (403.72,188) .. controls (366.38,213.13) and (285.66,233.5) .. (223.42,233.5) .. controls (161.18,233.5) and (141,213.13) .. (178.34,188) -- cycle ;
\draw  [color={rgb, 255:red, 255; green, 0; blue, 0 }  ,draw opacity=1 ] (192.03,132.15) .. controls (196.28,132.95) and (199.91,132.75) .. (202.89,131.57) .. controls (200.54,133.76) and (198.81,136.95) .. (197.72,141.13) ;
\draw  [color={rgb, 255:red, 255; green, 0; blue, 0 }  ,draw opacity=1 ] (266.65,106.66) .. controls (270.61,108.4) and (274.18,109.02) .. (277.36,108.53) .. controls (274.58,110.14) and (272.17,112.86) .. (270.17,116.69) ;
\draw  [color={rgb, 255:red, 255; green, 0; blue, 0 }  ,draw opacity=1 ] (350.97,93.94) .. controls (354.22,96.79) and (357.45,98.46) .. (360.63,98.95) .. controls (357.49,99.65) and (354.38,101.52) .. (351.31,104.57) ;
\draw  [color={rgb, 255:red, 255; green, 0; blue, 0 }  ,draw opacity=1 ] (427.92,125.95) .. controls (423.73,127.04) and (420.54,128.76) .. (418.34,131.11) .. controls (419.53,128.12) and (419.74,124.5) .. (418.94,120.25) ;
\draw  [color={rgb, 255:red, 255; green, 0; blue, 0 }  ,draw opacity=1 ] (359.44,170.89) .. controls (355.33,169.57) and (351.7,169.32) .. (348.59,170.13) .. controls (351.2,168.24) and (353.31,165.29) .. (354.91,161.27) ;
\draw  [color={rgb, 255:red, 255; green, 0; blue, 0 }  ,draw opacity=1 ] (291.89,188.94) .. controls (288.05,186.95) and (284.52,186.1) .. (281.32,186.38) .. controls (284.2,184.96) and (286.77,182.4) .. (289.02,178.7) ;
\draw  [color={rgb, 255:red, 255; green, 0; blue, 0 }  ,draw opacity=1 ] (227.89,195.18) .. controls (224.59,192.39) and (221.33,190.78) .. (218.15,190.34) .. controls (221.27,189.59) and (224.35,187.66) .. (227.36,184.56) ;
\draw  [color={rgb, 255:red, 255; green, 0; blue, 0 }  ,draw opacity=1 ] (172.64,188.01) .. controls (171.2,183.93) and (169.2,180.9) .. (166.68,178.92) .. controls (169.75,179.84) and (173.38,179.74) .. (177.55,178.57) ;
\draw  [color={rgb, 255:red, 0; green, 64; blue, 255 }  ,draw opacity=1 ] (166.4,216.14) .. controls (167.68,220.27) and (169.56,223.37) .. (172.01,225.45) .. controls (168.97,224.41) and (165.34,224.37) .. (161.13,225.37) ;
\draw  [color={rgb, 255:red, 0; green, 64; blue, 255 }  ,draw opacity=1 ] (218.22,228.72) .. controls (221.34,231.72) and (224.48,233.53) .. (227.63,234.17) .. controls (224.46,234.72) and (221.27,236.45) .. (218.06,239.35) ;
\draw  [color={rgb, 255:red, 0; green, 64; blue, 255 }  ,draw opacity=1 ] (284.17,223.85) .. controls (287.86,226.11) and (291.32,227.2) .. (294.53,227.15) .. controls (291.56,228.37) and (288.81,230.74) .. (286.31,234.27) ;
\draw  [color={rgb, 255:red, 0; green, 64; blue, 255 }  ,draw opacity=1 ] (352.65,206.65) .. controls (356.61,208.39) and (360.19,209.02) .. (363.36,208.54) .. controls (360.58,210.14) and (358.17,212.86) .. (356.16,216.69) ;
\draw  [color={rgb, 255:red, 0; green, 64; blue, 255 }  ,draw opacity=1 ] (411.82,173.2) .. controls (416.14,173.36) and (419.69,172.62) .. (422.47,170.99) .. controls (420.47,173.51) and (419.24,176.93) .. (418.79,181.23) ;
\draw  [color={rgb, 255:red, 0; green, 64; blue, 255 }  ,draw opacity=1 ] (362.76,148.3) .. controls (359.53,145.43) and (356.32,143.74) .. (353.14,143.22) .. controls (356.29,142.55) and (359.41,140.7) .. (362.5,137.67) ;
\draw  [color={rgb, 255:red, 0; green, 64; blue, 255 }  ,draw opacity=1 ] (286.39,153.01) .. controls (282.3,151.61) and (278.68,151.3) .. (275.56,152.05) .. controls (278.2,150.21) and (280.36,147.3) .. (282.03,143.31) ;
\draw  [color={rgb, 255:red, 0; green, 64; blue, 255 }  ,draw opacity=1 ] (215.89,172.42) .. controls (211.58,172.01) and (207.99,172.54) .. (205.13,174) .. controls (207.27,171.6) and (208.7,168.26) .. (209.39,164) ;

\draw (195,104.3) node [anchor=north west][inner sep=0.75pt]  [font=\footnotesize,color={rgb, 255:red, 255; green, 0; blue, 0 }  ,opacity=1 ]  {$\overline{c} ( t)$};
\draw (193,244.3) node [anchor=north west][inner sep=0.75pt]  [font=\footnotesize,color={rgb, 255:red, 0; green, 0; blue, 255 }  ,opacity=1 ]  {$c( t)$};
\draw (134,187.3) node [anchor=north west][inner sep=0.75pt]  [font=\scriptsize]  {$\mathnormal{x}( t)$};

\end{tikzpicture}
    \caption{Illustration of a mirrored curve on $SM$.}
    \label{fig:mirror-geodesic}
\end{figure}

\end{center}
\subsection{Maslov index}\label{subsection:maslov-index}

Let $\mathbb{P}(\Sigma)$ denote the \textit{projective bundle} (also called the \textit{Grassmann $1$-plane bundle}) of the vector bundle $\Sigma\to SM$. That is, $\mathbb{P}(\Sigma)$ is the $4$-dimensional manifold obtained by projectivizing $\Sigma(v)$ for each $v\in SM$: the fibre over $v$ consists of all $1$-dimensional subspaces of $\Sigma(v)$. Note that both $\mathbb{V}^*$ and $\mathbb{H}^*$ define sections of this bundle. The flow $\{\tilde{\varphi}_t\}_{t\in \R}$ on $T^*(SM)$ naturally induces a flow on $\mathbb{P}(\Sigma)$, which we continue to denote by the same symbol.

Another concept which is usually defined on $T(SM)$, but we prefer to 
formulate on $T^*(SM)$, is the Maslov index (see \cite[Section 2.4]{paternain99}). For us, this is a way to attach an index to every continuous closed
curve $c:\mathbb{S}^1\to \mathbb{P}(\Sigma)$.

For every $(v,E^*) \in \mathbb{P}(\Sigma)$ with $E^* \neq \mathbb{H}^*(v)$, there exists a unique real number $r(v,E^*) \in \mathbb{R}$ such that $$E^* = \mathbb{R}(r(v,E^*)\beta(v)+  \psi_\lambda(v)).$$
Using this identification, define the map $m : \mathbb{P}(\Sigma) \rightarrow \mathbb{S}^1$ by 
\[m(v, E^*) \coloneqq \begin{cases} \dfrac{1+i r(v,E^*)}{1-ir(v,E^*)} & \text{if } E^* \neq \mathbb{H}^*(v), \\ 
-1 & \text{otherwise.} 
\end{cases}\]
Let $\widehat{\mathbb{R}} \coloneqq \mathbb{R} \cup \{\infty\}$ denote the one-point compactification of $\mathbb{R}$. We extend $r$ to $\widehat{\mathbb{R}}$ by setting $r(v, \mathbb{H}^*(v)) \coloneqq\infty$. This extension $r : \mathbb{P}(\Sigma) \rightarrow \widehat{\mathbb{R}}$ is continuous when $\mathbb{P}(\Sigma)$ is equipped with the Grassmannian topology, and it is easy to check that the resulting map $m: \mathbb{P}(\Sigma)\to \mathbb{S}^1$ is also continuous. See Figure \ref{figure:m-2}.

\begin{figure}[h]
\centering 
\begin{subfigure}[t]
{0.5\textwidth}
\centering
\includegraphics[scale=0.32]{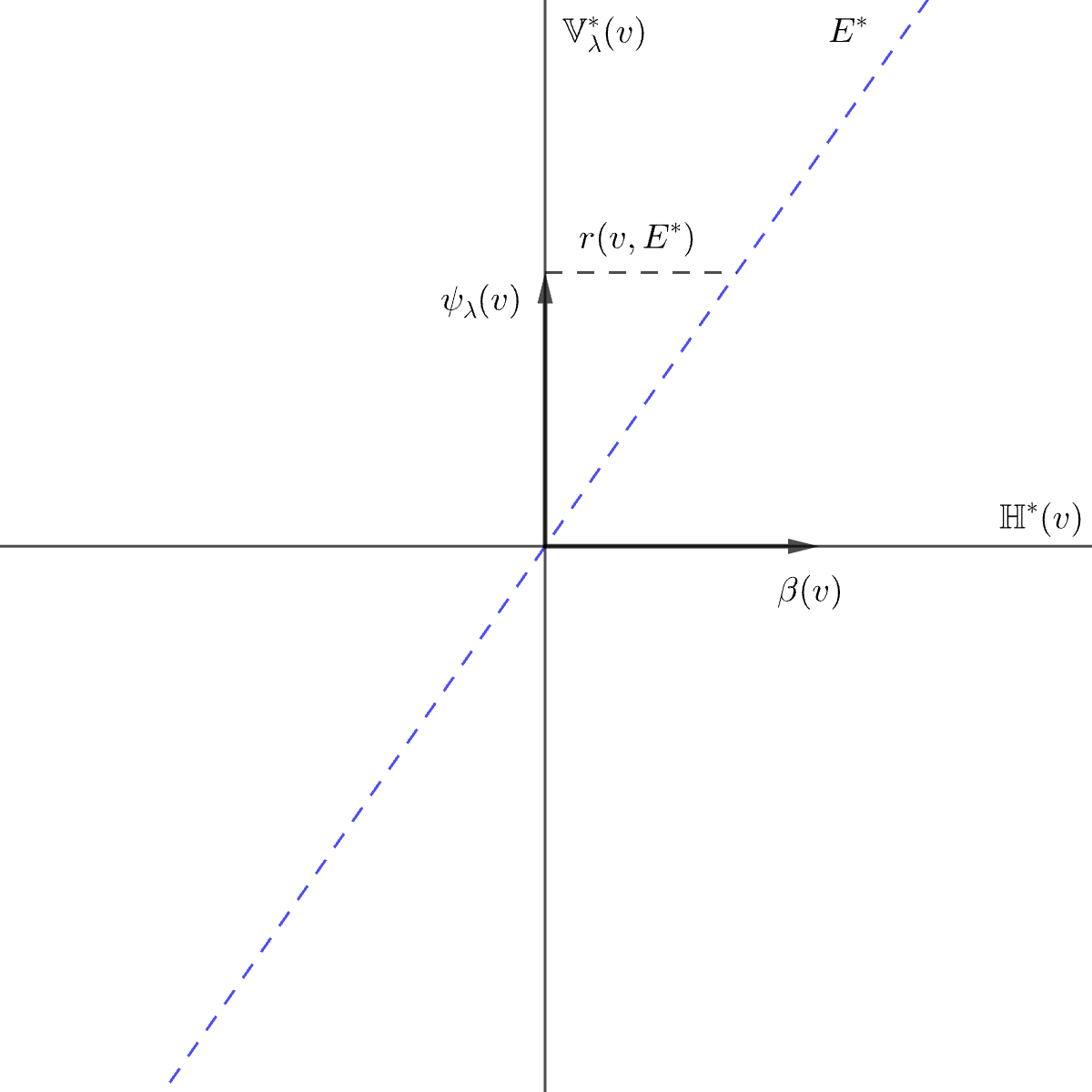}
\end{subfigure}%
\begin{subfigure}[t] 
{0.5\textwidth}
\centering
\includegraphics[scale=0.3392226148]{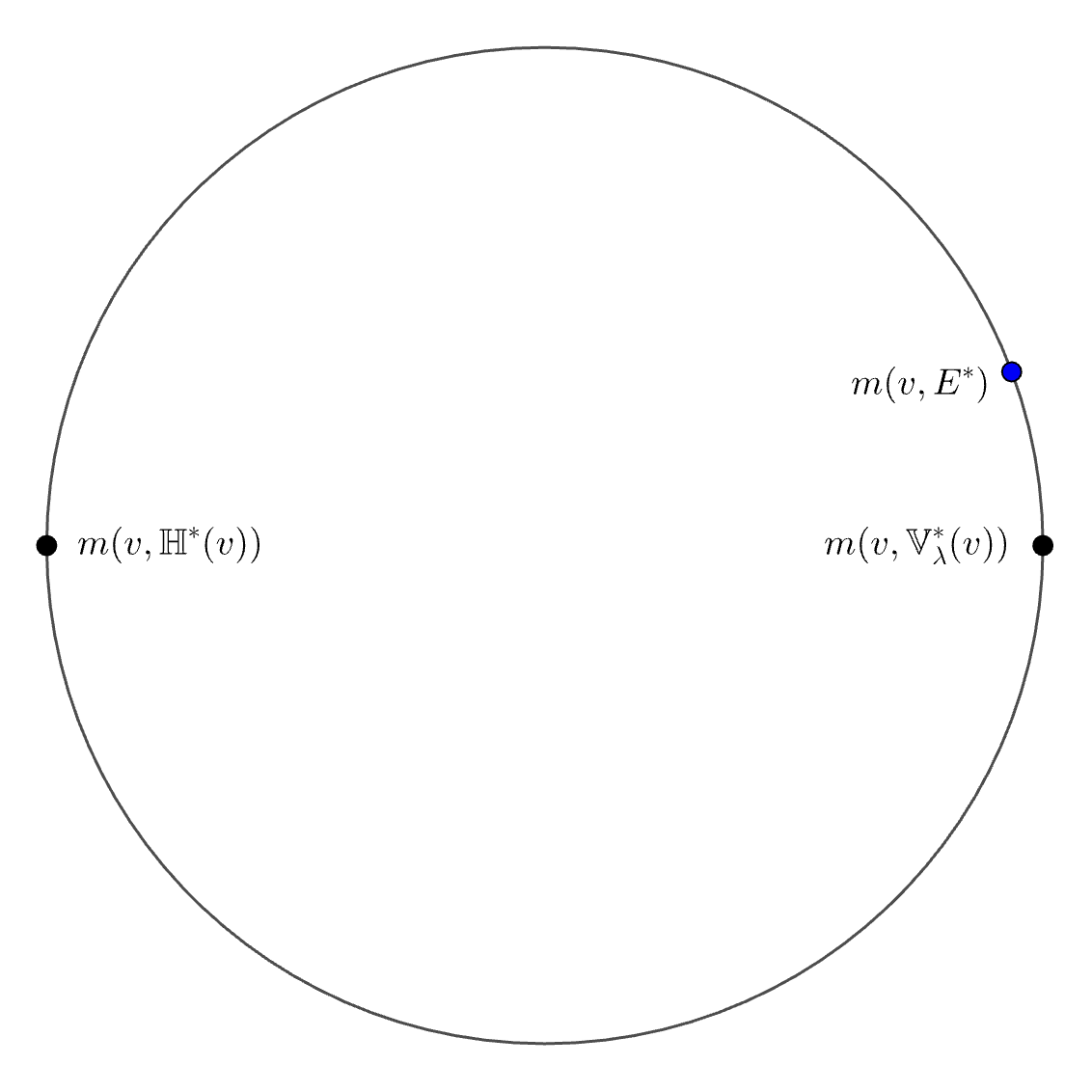}
\end{subfigure}
\caption[The map $m: \mathbb{P}(\Sigma)\to \mathbb{S}^1$ on a fibre over $v\in SM$]{The map $m: \mathbb{P}(\Sigma)\to \mathbb{S}^1$ on a fibre over $v\in SM$.}
\label{figure:m-2}
\end{figure}

The \emph{Maslov index} $\mathfrak{m} : \pi_1(\mathbb{P}(\Sigma)) \rightarrow \mathbb{Z}$ is defined by
\[ \mathfrak{m}([c]) \coloneqq \deg(m \circ c)\]
for every continuous closed curve $c : \mathbb{S}^1\to \mathbb{P}(\Sigma)$. 

Now observe that any $1$-dimensional subbundle $E^*\subset \Sigma$ naturally defines a section $E^* : SM \rightarrow \mathbb{P}(\Sigma)$.  The \emph{$E^*$-Maslov index} $\nu_{E^*} : \pi_1(SM) \rightarrow \mathbb{Z}$ is then defined by 
\begin{equation}\label{eq:maslov-index}
 \nu_{E^*}([c]) \coloneqq \mathfrak{m}([E^* \circ c]) = \deg(m \circ E^* \circ c)
\end{equation}
for every continuous closed curve $c: \mathbb{S}^1\to SM$.

\subsection{Closed orbits near $\mathbb{H}^*$}

In what follows, let $E^*: SM\to \mathbb{P}(\Sigma)$ be a continuous flow-invariant section.

\begin{definition} A continuous closed curve $c : \mathbb{S}^1\to SM$ is a \textit{closed orbit near $\mathbb{H}^*$} with respect to $(M,g,\lambda, E^*)$ if for all $t_0\in \mathbb{S}^1$ where $E^*(c(t_0))=\mathbb{H}^*(c(t_0))$,
there exists $\varepsilon>0$ such that for $t\in (-\varepsilon, \varepsilon)$ we have
$c(t_0+t)=\varphi_t(c(t_0)).$
\end{definition}

Intuitively, a closed orbit $c$ near $\mathbb{H}^*$ can do whatever it wants whenever $E^*(c(t))$ is not $\mathbb{H}^*$. However, near intersections with $\mathbb{H}^*$, the curve $c$ coincides with a closed orbit. 
Obviously, a closed orbit is a closed orbit near $\mathbb{H}^*$ for any section $E^*$. Note that this aligns with the definition of a pseudo-geodesic found in \cite{paternain94}.

\begin{lemma}\label{lemma:maslov-index-counting}
    If $c: \mathbb{S}^1\to SM$ is a closed orbit near $\mathbb{H}^*$ with respect to $(M, g, \lambda, E^*)$, then $\nu_{E^*}([c])\geq 0$. Moreover, $\nu_{E^*}([c])>0$ if and only if there exists $t\in \mathbb{S}^1$ such that 
$E^*(c(t))=\mathbb{H}^*(c(t)).$
\end{lemma}

\begin{proof}
    Let
    $P\coloneqq\{t\in \mathbb{S}^1 \mid E^*(c(t))=\mathbb{H}^*(c(t))\}.$
    Since $c$ is a closed orbit near $\mathbb{H}^*$ for $(M, g,\lambda, E^*)$ and $E^*$ is flow-invariant, $E^* \circ c : \mathbb{S}^1\to \mathbb{P}(\Sigma)$ is differentiable in a neighbourhood of any point $t\in P$. The twist property of the cohorizontal subbundle \cite[Lemma 2.1]{echevarria-cuesta25b} then ensures that $\nu_{E^*}([c])\geq 0$. In fact, the $E^*$-Maslov index $\nu_{E^*}([c])$ is equal to the cardinality of $P$.
\end{proof}


    



\begin{proposition}\label{proposition:existence-of-pseudo-geodesics}
    Suppose that there exists $v\in \Omega$ such that $E^*(v)=\mathbb{H}^*(v)$. Then, there exists a closed orbit $c: \mathbb{S}^1\to SM$ near $\mathbb{H}^*$ with $\nu_{E^*}([c])>0$. If the thermostat is reversible, then $c$ can be chosen so that $\overline{c}$ is also a closed orbit near $\mathbb{H}^*$.
\end{proposition}

\begin{proof}
We follow the proof in \cite[Proposition 4]{paternain94} (see also \cite[Lemma 2.49]{paternain99}). Let $v\in \Omega$ be a point such that $E^*(v)=\mathbb{H}^*(v)$. By 
\cite[Lemma 2.1]{echevarria-cuesta25b},
there exists $\varepsilon>0$ such that, for all $t\in (-2\varepsilon, 2\varepsilon)$ with $t\neq 0$, we have $E^*(\varphi_t(v))\neq \mathbb{H}^*(\varphi_t(v))$. Let $U_1$ be a neighbourhood of $\varphi_{-\varepsilon}(v)$ such that $E^*(\varphi_t(v'))\neq \mathbb{H}^*(\varphi_t(v'))$ for all $v'\in U_1$. Analogously, let $U_2$ be a neighbourhood of $\varphi_{\varepsilon}(v)$ such that $E^*(\varphi_t(v'))\neq \mathbb{H}^*(\varphi_t(v'))$ for all $v'\in U_2$. Since $v\in \Omega$, there exists $T>0$ and $w\in U_1$ such that $\varphi_{T}(w)\in U_2$. Connect $\varphi_\varepsilon(v)$ with $w$ by a path $c_1$ contained in $U_1$, and $\varphi_{T}(w)$ with $\varphi_{-\varepsilon}(v)$ by a path $c_2$ contained in $U_2$. The path $c$ obtained by joining the arcs $c_1$, $\varphi_t(v)|_{[-\varepsilon, \varepsilon]}$, $c_2$ and $\varphi_t(w)|_{[0, T]}$ gives us the desired result thanks to Lemma \ref{lemma:maslov-index-counting}.

If the thermostat is reversible, the mirrored curve of any closed orbit is again a closed orbit, so our construction guarantees that $\overline{c}$ is also a closed orbit near $\mathbb{H}^*$.
\end{proof}

\subsection{Proof of Theorem \ref{theorem:dominated-splitting-theorem}} At this stage, we have the ingredients to show the following result for reversible thermostats.

\begin{proposition}\label{proposition:invariant-subbundle}
    Let $(M, g, \lambda)$ be a reversible thermostat. If there exists a continuous  $1$-dimensional flow-invariant subbundle $E^*$ of its characteristic set, then we have $E^*(v)\cap\mathbb{H}^*(v)=\{0\}$ for all $v\in \Omega$. Furthermore, the thermostat does not have conjugate points in its non-wandering set $\Omega$.
\end{proposition}

\begin{proof}
   Given our characterization of conjugate points in \cite[Theorem 1.5]{echevarria-cuesta25b}, we only have to show that $E^*(v)\cap \mathbb{H}^*(v)=\{0\}$ for all $v\in \Omega$. Suppose, for contradiction, that $E^*(v) \cap \mathbb{H}^*(v) \neq \{0\}$ for some $v\in \Omega$. By Proposition \ref{proposition:existence-of-pseudo-geodesics}, there exists a closed orbit $c$ near $\mathbb{H}^*$ through $v$ such that $\nu_{E^*}([c])>0$. Let $c^{-1}(t)\coloneq c(-t)$ and observe that $c^{-1}$ and $\overline{c}$ are homotopic. Thus $0<\nu_{E^*}([c])=-\nu_{E^*}([c^{-1}])=-\nu_{E^*}([\overline{c}])$. However, since $\overline{c}$ is also a closed orbit near $\mathbb{H}^*$, it has a non-negative $E^*$-Maslov index by Lemma \ref{lemma:maslov-index-counting}, which is a contradiction. 
\end{proof}


 As for any projectively Anosov flow, the \textit{dual stable} and \textit{unstable subbundles} $E_{s}^*, E_u^*$ in $ T^*(SM)$ are defined by the relationship $$ E_{s}^*(q^{-1}(\mathcal{E}_s)) = 0=E_{u}^*(q^{-1}(\mathcal{E}_u)),$$ where $q : T(SM) \rightarrow T(SM)/\mathbb{R} F$ is the quotient map. These two subbundles are continuous (see \cite[Lemma 2.28]{araujo10}) and flow-invariant, so
 we obtain Theorem \ref{theorem:dominated-splitting-theorem} by applying Proposition \ref{proposition:invariant-subbundle} to either of them.
 While it could still be the case that Theorem \ref{theorem:dominated-splitting-theorem} applies to non-reversible thermostats, we include the following simple example to show that Proposition \ref{proposition:invariant-subbundle} fails for these systems.

\begin{lemma}\label{lemma:counterexample}
Let $(\mathbb{T}^2, g, \lambda_0)$ be a magnetic system, where $g$ is a flat Riemannian metric and $\lambda_0 \neq 0$ is constant. Then, there exist smooth $1$-dimensional flow-invariant subbundles of $\Sigma$, yet the non-wandering set $\Omega=S\mathbb{T}^2$ contains conjugate points.
\end{lemma}

\begin{proof}

Note that $\mathbb{K}=\lambda_0^2$. By \cite[Lemma 2.1]{echevarria-cuesta25b}, the functions $x,y\in \mathcal{C}^\infty(\R, \R)$ defined in  equation \eqref{eq:equations-of-motion} satisfy
\begin{equation*}\label{eq:flow-on-sigma}
    \begin{cases}
\dot{x}- \lambda_0^2 y=0,\\
\dot{y}+x = 0,
    \end{cases}
\end{equation*}
with solutions
\begin{equation*}\label{eq:flow-on-sigma2}
    \begin{cases}
x(t)= x(0)\cos(\lambda_0 t)+y(0)\lambda_0 \sin (\lambda_0 t),\\
y(t)= y(0)\cos(\lambda_0 t) - x(0) \lambda_0^{-1}\sin (\lambda_0 t).
    \end{cases}
\end{equation*}
On the other hand, using coordinates $(z, \theta)\in S\mathbb{T}^2$, the equations of motion for the thermostat give us $\dot{\theta}(t)=\lambda_0$. Therefore, for any $c_1, c_2\in \R$, the smooth $1$-forms 
$$\lambda_0 (c_1\cos(\theta)+c_2 \sin (\theta)) \beta + (c_2\cos(\theta) - c_1\sin (\theta))\psi_\lambda$$
are flow-invariant. Finally, this is a magnetic flow, so the Liouville volume form is preserved and hence $\Omega=S\mathbb{T}^2$. However, by \cite[Theorem 1.3]{echevarria-cuesta25b}, there must be conjugate points.
\end{proof}

\begin{remark}
    One reason as to why Proposition \ref{proposition:invariant-subbundle} may fail when considering a non-reversible thermostat is that the characteristic set of a thermostat and its mirror are different, and so the homotopy between $c^{-1}$ and $\bar{c}$ would have to pass through the cotangent bundle $T^*(SM)$ and leave the characteristic set $\Sigma$. This causes issues with the calculation of the Maslov index.
\end{remark}


\bibliographystyle{alpha}
\bibliography{references}
 
\end{document}